\documentclass[a4paper,12pt]{article}
\usepackage{extsizes}
\usepackage{cmap}


\usepackage[a4paper, left=30mm, right=15mm, top=20mm, bottom=20mm]{geometry}
\usepackage{listings}
\usepackage[T2A]{fontenc}
\usepackage[utf8]{inputenc}
\usepackage[pdftex,unicode]{hyperref}
\usepackage{indentfirst}
\usepackage{amsfonts,amssymb,amsmath,amsthm}
\usepackage{colortbl}

\usepackage[square,numbers,sort&compress]{natbib}
\renewcommand{\geq}{\geqslant}
\renewcommand{\leq}{\leqslant}

\DeclareMathOperator{\conv}{conv}

\DeclareMathOperator{\vol}{vol}

\DeclareMathOperator{\erf}{erf}

\newtheorem{lemma}{Lemma}[section]
\newtheorem{theorem}{Theorem}[section]

\begin{document}
\setlength{\abovedisplayskip}{4pt plus 0pt minus 0pt}  % было: 8 + 2 - 4
\setlength{\belowdisplayskip}{3pt plus 0pt minus 0pt}
\begin{center}
\Large\bf Extremal  convex bodies for affine measures of symmetry
\end{center}
\begin{center}
Safronenko Evgenii
\end{center}
\begin{center}
{\bf Abstract}
\end{center}
This paper is devoted to measures of symmetry based on distance between centroid and one of the centers of John and L\"owner ellipsoid. The author proves the accuracy of the derived upper bounds for the considered measures of symmetry.

\section{Introduction}
By convex body or simply a body, we shall mean a compact convex subset of $\mathbb{R}^n$ with non-empty interior. Also we denote by $\mathcal{K}_n$ the set of all convex bodies in $\mathbb{R}^n.$ For any convex body $K\in\mathcal{K}_n$ we call the point $p(K)\in K$ affine invariant if for any nonsingular affine map $T:\mathbb{R}^n\to\mathbb{R}^n$ we have 
$$
p(T(K))=T(p(K)).
$$

In 2011 in paper \cite{MMCSEW} M. Meyer, C. Sch\"utt and E.M. Werner  defined measure of asymmetry for any convex body $K\in\mathcal{K}_n$ by following: for fixed two affine invariant points $p_1(K),p_2(K)\in K$ as quantity~$d$ equal
\begin{align*}
&d(p_1(K),p_2(K))=0,\quad\text{if } p_1(K)=p_2(K)\\
\text{and}\qquad\qquad\qquad&\\
&d(p_1(K),p_2(K))=\frac{\|p_1(K)-p_2(K)\|}{\vol_1(l\cap K)},\quad\text{if } p_1(K)\ne p_2(K),\qquad\qquad\qquad\qquad
\end{align*}
where $l$ --- line, through $p_1(K)$ and $p_2(K)$. Corresponding measure of symmetry is the map defined on the set $\mathcal{K}_n$ by
$$
K\to\varphi_{p_1,p_2}(K)=1-d(p_1(K),p_2(K)).
$$

The quantity $d(p_1(K), p_2(K))$ is affine invariant and takes values from $0$ to $1.$ If the convex body $K$ is centrally symmetric then for all affine invariant points $p_1(K)$ and $p_2(K)$ we have $d(p_1(K),p_2(K))=0.$ Centrally symmetric bodies are not the only bodies with this property. For example, the measure of asymmetry for simplex will also be $0$ for any pair of affine invariant points.

This article is devoted to study dependence of $n$ the quantity 
$$
\max_{K\in\mathcal{K}_n}d(p_1(K),p_2(K)).
$$
for the pairs of some classic affine invariant points such as centroid g(K), centers of the John and the L\"owner ellipsoids of convex body $K\in\mathcal{K}_n.$

In paper \cite{MMCSEW} authors consider centroid $g(K)$ and Santal\'o point $s(K).$ They give estimate
\begin{align}
d(g(K),s(K))\leq1-\frac2{n+1},
\end{align}
but with the little mistake in the argument. After that in article \cite{Mordhorst} O. Mordhorst give the correct proof of this fact. The main result of the paper \cite{MMCSEW} is construction of convex bodies $Q_n\in\mathcal{K}_{n+1}$ for which the values of the measures of asymmetry $d(g(Q_n),s(Q_n))$ are separated from zero
\begin{align*}
\frac1e\frac{\sqrt{e\pi}-2}{\sqrt{e\pi}+\frac2{e-1}}&\leq\varliminf_{n\to\infty}d(g(Q_n),s(Q_n))\leq\\
&\leq\varlimsup_{n\to\infty}d(g(Q_n),s(Q_n))\leq\frac1e\frac{\sqrt{e\pi}-1}{\sqrt{e\pi}+\frac1{e-1}}.
\end{align*}
Also in the same article there is an example of convex bodies $G_n\in\mathcal{K}_{n+1}$ for whose centers of the John and L\"owner ellipsoids can be "far away"
\begin{align}
\label{PastResult}
d(j(G_n),l(G_n))\xrightarrow[n\to\infty]{}\frac12.
\end{align}
In 2017 O. Mordhost \cite{Mordhorst} provided the upper bound true for arbitrary convex body $K\in\mathcal{K}_n$
\begin{align}
\label{MordhorstJL}
d(j(K),l(K))\leq1-\frac2{n+1}.
\end{align}
Also he proved that this inequality is asymptoically sharp up to the constant.

Using well-known inclusions \eqref{JohnWellKnown},\eqref{LownerWellKnown} and \eqref{CentroidWellKnown}, by similar arguments from \cite{Mordhorst}, it is easy to show 
\begin{align}
\label{aprioryint3}
d(g(K),j(K))&\leq1-\frac2{n+1},\\
\label{aprioryint2}
d(g(K),l(K))&\leq1-\frac2{n+1}.
\end{align}
for any convex body $K\in\mathcal{K}_n.$

In the section 6 will be given construction of convex bodies $F_n,W_n\in\mathcal{K}_n$ for which estimates \eqref{aprioryint3}  and \eqref{aprioryint2} are asymptotically sharp up to the order $n$
\begin{align*}
d(g(F_n),j(F_n))=1-\frac1nC^{*}+o\left(\frac1n\right)
\end{align*}
and
\begin{align*}
d(g(W_n),l(W_n))=1-\frac1nC^{**}+o\left(\frac1n\right),
\end{align*}
where $C^*\approx 13$ and $C^{**}\approx 20$.

In order to show the universality of given approach we shall also construct convex bodies $M_n\in\mathcal{K}_n$ for whose estimation for measure of symmetry \eqref{MordhorstJL} is exact
$$
d(j(M_n),l(M_n))=1-\frac8{n}+o\left(\frac1n\right)
$$
when $n\to\infty$.

Thus, in this article will be proven estimates asymptotically sharp up to the order $n$
$$
\frac2{n+1}\leq\varphi_{g,j}\leq1,\quad\frac2{n+1}\leq\varphi_{g,l}\leq1\quad\text{and}\quad\frac2{n+1}\leq\varphi_{j,l}\leq1.
$$
\medbreak
%%%%%%%%%%%%%%%%%%%%%%%%%%%%%%%%%%%%%%%%%%%%%%%%
%%                                                          ОБОЗНАЧЕНИЯ
%%%%%%%%%%%%%%%%%%%%%%%%%%%%%%%%%%%%%%%%%%%%%%%%
\section{Definitions and auxiliary results}
For any $x\in\mathbb{R}^n$ for $1\leq k\leq n$ we denote by $x^k$ the $k$-th coordinate, its scalar product with $y\in\mathbb{R}^n$ by $\langle x,y\rangle =\sum_{k=1}^nx^ky^k$ and corresponding euclidean norm of $x$ by $\|x\|= \sqrt{\langle x,x\rangle}$. Convex hull of given subsets $A$ and $B$ of $\mathbb{R}^n$ is denoted by $\conv(A,B)$.

For unit vector $\xi\in S^{n-1}=\{x\in\mathbb{R}^n : \|x\|=1\}$ define its corresponding orthogonal subspace 
$$
\xi^\perp=\{x\in\mathbb{R}^n : \langle x,\xi\rangle=0\}.
$$
Fix number $t\in\mathbb{R}$, unit vector $\xi\in S^{n-1}$ and convex body $K\in\mathcal{K}_n.$ We call corresponding $(n-1)$-dimensional section of $K$, orthogonal to the direction $\xi$ and passing through the point $t\xi$
$$
K(\xi,t)=\{x\in K : \langle x,\xi\rangle=t\}= K\cap\xi^\perp.
$$
We shall be interested in cross sections orthogonal to the direction $e_n=(0,\ldots,0,1)$. We shall denote them by $K(t)=K(e_n,t)$ and identify with $(n-1)$-dimensional convex body in $\mathbb{R}^{n-1}$. 

We introduce the following notation for standard convex bodies, which we need for future constructions: $B^n_1=\{x\in\mathbb{R}^n: \sum_{k=1}^n|x^k|\leq1\}$ is $n$-dimensional octahedron, $B_2^n=\{x\in\mathbb{R}^n: \|x\|\leq1\}$ is euclidean ball, $B_\infty^n=\{x\in\mathbb{R}^n: \underset{k}{\max}|x^k|\leq1\}$ is $n$-dimensional cube, and $\Delta_n$ is $n$-dimensional simplex, inscribed into the ball $B^n_2$.
Denote by $\vol_n$ $n$-dimensional Lebesgue measure. $\partial K$~is~the boundary of convex body $K$.

It is well-known (see e.g. \cite{{Zaguskin}, {Danzer}}), that for any convex bodies $K\in\mathcal{K}_n$ there exists inscribed ellipsoid of maximal volume $J(K)$ and described ellipsoid $L(K)$ of minimal volume. This ellipsoids are solutions of corresponding extremal problems
\begin{align*}
\vol_n(J(K))=\max_{\mathcal{E}\subset K}\vol_n(\mathcal{E}),\\
\vol_n(L(K))=\min_{K\subset\mathcal{E}}\vol_n(\mathcal{E}),
\end{align*}
where $\mathcal{E}$ is ellipsoid. We say that convex body is in John position if the ball $B^n_2$ is John ellipsoid. Centers of John ellipsoid $J(K)$ and L\"owner ellipsoid $L(K)$, respectively $j(K)$ and $l(K)$, are affine invariant points. 

We need the next lemma to compute centers of John ellipsoid  of  our preliminary constructions.
%%%%%%%%%%%%%%%%%%%%%%%%%%%%%%%%%%%%%%%%%%%%%%%%
%%                                                        Лемма Матью Майера
%%%%%%%%%%%%%%%%%%%%%%%%%%%%%%%%%%%%%%%%%%%%%%%%
\begin{lemma}[M. Meyer, C. Schütt, E. M. Werner \cite{MMCSEW}]
\label{wellknown}
Let for some positive number $\alpha>0$ the ball $\alpha B^n_2$ is John ellipsoid (respectively L\"owner ellipsoid) of convex body $K\in\mathbb{R}^n$. Then for any $s,t\geq0$ the John ellipsoid (respectively L\"owner ellipsoid) of convex body $sK+tB^n_2$ is the ball $(s\alpha+t)B^n_2$.
\end{lemma}

Another example of affine invariant points are centroid $g(K)$ and Santal\'o point $s(K)$ of convex body $K\in\mathcal{K}_n$ \cite{MSW1,MSW2}
\begin{align*}
&g(K)=\frac1{\vol_n(K)}\int_Kx\ dx\qquad\qquad\qquad\qquad\qquad\qquad\qquad\\
\text{and}\qquad\qquad\qquad\qquad\qquad\qquad\qquad&\\
&\vol_n\left(K^{s(K)}\right)=\min_{x\in K}\vol_n(K^{x}),
\end{align*}
where $K^x$ is the polar of convex body $K$ relative to the point $x$
$$
K^x=\{y\in\mathbb{R}^n : \langle z-x,y\rangle\leq1 \text{ for all }z\in K \}.
$$

The next inclusions are well-known properties (see e.g.\cite[chapter 3]{Pisier} and \cite[\S7]{BONNESEN}) of corresponding affine invariant points
\begin{align}
\label{JohnWellKnown}
J(K)-j(K)&\subset K-j(K)\subset n(J(K)-j(K)),\\
\label{LownerWellKnown}
\frac1n(L(K)-l(K))&\subset K-l(K)\subset L(K)-l(K),\\
\label{CentroidWellKnown}
K-g(K)&\subset n(g(K)-K),
\end{align}
for $K\in\mathcal{K}_n.$

From Minkowski theorem (see e.g. \cite[\S 4.2]{Shcneider} or \cite[chapter 6 \S1]{Hadwiger}) follows that for all convex bodies $K,$ $L\in\mathcal{K}_n$ and non-negative numbers $\lambda,\mu\geq0$ the volume of body $\lambda K+\mu L$ is polynomial of degree $n$ of $\lambda$ and $\mu$
$$
\vol_n(\lambda K+\mu L)=\sum_{k=0}^nC_n^k\lambda^{n-k}\mu^kV_{n-k,k}(K,L).
$$
Coefficients $V_{n-k,k}(K,L)$ are called mixed volumes. They are homogeneous of degree $n-k$ and $k$ in $K$ and $L$ respectively.
 
To calculate centroids of preliminary convex bodies in the section $4$ we need lemma.
%%%%%%%%%%%%%%%%%%%%%%%%%%%%%%%%%%%%%%%%%%%%%%%%
%%                                                       ЛЕММА О ЦЕНТРЕ МАСС
%%%%%%%%%%%%%%%%%%%%%%%%%%%%%%%%%%%%%%%%%%%%%%%%
\begin{lemma}[M. Meyer, C. Schütt, E. M. Werner \cite{MMCSEW}]
\label{LemmaMixedVol}
Let $K,L\in\mathcal{K}_n$ be convex bodies, $c>0$ is positive number. Consider convex body $P_n=\conv\left(\left(K,0\right),\left(L,c\right)\right).$ Then $(n+1)$-th coordinate of centroid $g^{n+1}(P_n)$ of convex body $P_n$ satisfy
\begin{align}
\label{CentroidForm}
g^{n+1}(P_n)=\frac{c}{n+2}\frac{\sum_{k=0}^n(k+1)V_{n-k,k}(K,L)}{\sum_{k=0}^nV_{n-k,k}(K,L)}.
\end{align}
\end{lemma}

Also we need mixed volumes $V_{n-k,k}\left(B^n_1, B^n_\infty\right)$ and $V_{n-k,k}\left(B^n_2, B^n_\infty\right).$ They are easy calculated by next lemma.
%%%%%%%%%%%%%%%%%%%%%%%%%%%%%%%%%%%%%%%%%%%%%%%%
%%                                                      ЧИТ-ЛЕММА
%%%%%%%%%%%%%%%%%%%%%%%%%%%%%%%%%%%%%%%%%%%%%%%%
\begin{lemma}[A. Pajor {\cite[theorem 1.10]{Pajor2}}, {{\cite[theorem 6]{Pajor1}}}]
\label{PajorTH}
For any convex body $K\in\mathcal{K}_n$ mixed volumes $V_{n-k,k}(K,B^n_\infty)$ satisfy
$$
V_{n-k,k}(K,B^n_\infty)=\frac{2^k}{C^k_n}\Biggl(\underset{|I|=n-k}{\sum_{I\subset\{1,\ldots,n\}}}\vol_nP^IK\Biggr),
$$
where for subset $I\subset\{1,\ldots,n\}, |I|=n-k$ by $P^I$ denoting projection from space $\mathbb{R}^n$ into $(n-k)$ dimensional subspace $\mathbb{R}^I$.
%$$
%P^IK=\{ y=(y^{i_1},\ldots,y^{i_{n-k}})\in\mathbb{R}^{n-k} : \exists\ x=(x^1,\ldots, x^n)\in K\ \forall\ i_k\in I\ x^{i_k}=y^{i_k}\}.
%$$
\end{lemma}
Remark that $P^IB^n_1=B^{n-k}_1$ and $P^IB^n_2=B^{n-k}_2$ for all subset $I\subset\{1,\ldots,n\}, |I|=n-k$. Therefore
\begin{align*}
V_{n-k,k}\left(B^n_1, B^n_\infty\right)=2^k\vol_{n-k}(B^{n-k}_1)\quad\text{and}\quad V_{n-k,k}\left(B^n_2, B^n_\infty\right)=2^k\vol_{n-k}(B^{n-k}_2).
\end{align*}

From homogeneity of mixed volumes $V_{n-k,k}$ we have
\begin{align}
\label{MixVolForm2}
V_{n-k,k}\left(\lambda B^n_1, \mu B^n_\infty\right)&=(2\mu)^k\lambda^{n-k}\vol_{n-k}(B^{n-k}_1)\\
\text{and}\qquad\qquad\quad\qquad\qquad\qquad\qquad\qquad\qquad\quad&\qquad\qquad\qquad\qquad\qquad\qquad\qquad\qquad\qquad\qquad\quad\notag\\
\label{MixVolForm}
V_{n-k,k}\left(\lambda B^n_2, \mu B^n_\infty\right)&=\left(2\mu\right)^k\lambda^{n-k}\vol_{n-k}(B^{n-k}_2).
\end{align}
for any $\lambda, \mu>0.$

We shall use notation $a\asymp b$ when there exists such non-negative numbers $c_1,c_2>0$ such that
$
c_1a\leq b\leq c_2a.
$

It is well-known that
\begin{align}
\label{volBn1}
\vol_n(B^n_1)=\frac{2^n}{\Gamma(1+n)}=\frac{2^n}{n!}\asymp\left(\frac{2e}{n}\right)^n\frac1{\sqrt{n}}
\end{align}
and
\begin{align}
\label{BallVoll}
\vol_n(B_2^n)&=\frac{\pi^\frac{n}2}{\Gamma\left(1+\frac{n}2\right)}\asymp\frac{(2e\pi)^\frac{n}2}{n^\frac{n+1}2},\qquad\text{when $n\geq1$}.
\end{align}
\medbreak
\section{Affine invariant points of cartesian product of convex bodies}
To prove lemma \ref{BodyProd} we need John theorem
(see e.g. \cite{KeithBall}).
%%%%%%%%%%%%%%%%%%%%%%%%%%%%%%%%%%%%%%%%%%%%%%%%
%%                                                          ТЕОРЕМА ДЖОНА
%%%%%%%%%%%%%%%%%%%%%%%%%%%%%%%%%%%%%%%%%%%%%%%%
\begin{theorem}[F. John \cite{JohnGeom}]
\label{JohnTH}
Let $K\subset\mathbb{R}^n$ be convex body.

The ball $B_2^n$ is John ellipsoid if and only if  $B_2^n\subset K$ and there are vectors  $\{u_i\}_{i=1}^m\subset\partial K\cap\partial B_2^n$ and positive numbers $\{c_i\}_{i=1}^m$ which satisfy conditions 
\begin{align*}
\sum_{i=1}^m c_iu_i&=0\\
\text{and}\qquad\qquad\qquad\qquad\qquad\qquad\qquad\qquad\qquad&\\
\sum_{i=1}^m c_i\langle x,u_i\rangle^2=\|x\|^2&\quad\text{for all }x\in\mathbb{R}^n.\qquad\qquad\qquad\qquad\qquad\qquad
\end{align*}

The ball $B_2^n$ is L\"owner ellipsoid if and only if $K\subset B_2^n$ and there are vectors $\{v_i\}_{i=1}^m\subset\partial K\cap\partial B_2^n$ and positive numbers $\{b_i\}_{i=1}^m$which satisfy conditions 
\begin{align*}
\sum_{i=1}^m b_iv_i&=0,\\
\text{and}\qquad\qquad\qquad\qquad\qquad\qquad\qquad\qquad\qquad&\\
\sum_{i=1}^m b_i\langle x,v_i\rangle^2=\|x\|^2\quad&\text{for all }x\in\mathbb{R}^n.\qquad\qquad\qquad\qquad\qquad\qquad\quad
\end{align*}
\end{theorem}
%%%%%%%%%%%%%%%%%%%%%%%%%%%%%%%%%%%%%%%%%%%%%%%%
%%                                                          ЛЕММА О ЦЕНТРАХ ПРОИЗВЕДЕНИЯ ВЫПУКЛЫХ ТЕЛ
%%%%%%%%%%%%%%%%%%%%%%%%%%%%%%%%%%%%%%%%%%%%%%%%
\begin{lemma}
\label{BodyProd}
Let $D\subset\mathbb{R}^n$ and $K\subset\mathbb{R}^m$. Then centroid, centers of John and L\"owner ellipsoids and Santal\'o point of convex body $D\times K\subset\mathbb{R}^n\times\mathbb{R}^m$ satisfy identity
\begin{align*}
g(D\times K)&=(g(D),g(K)),\\
j(D\times K)&=(j(D),j(K)),\\
l(D\times K)&=(l(D),l(K))
\end{align*}
and
$$
s(D\times K)=(s(D),s(K)).
$$
\end{lemma}
\begin{proof}
Identity for centroid follows from Fubini's theorem.
%%%%%%%%%%%%%%%%%%%%%%%%%%%%%%%%%%%%%%%%%%%%%%%%
%%                                                          СЛУЧАЙ ЭЛЛИПСОИДА ДЖОНА
%%%%%%%%%%%%%%%%%%%%%%%%%%%%%%%%%%%%%%%%%%%%%%%%
\medbreak
Let us prove the statement for the center of the John ellipsoid $D\times K$.

Suppose that bodies $D$ and $K$ are in John's position. Lets consider for convex bodies $D$ and $K$  corresponding vectors $\{u_i\}_{i=1}^k\subset\partial D\cap\partial B^n_2$ and $\{w_j\}_{j=1}^s\subset\partial K\cap\partial B^m_2$ and positive numbers $\{c_i\}_{i=1}^k$ and $\{d_j\}_{j=1}^s$ satisfying conditions from theorem \ref{JohnTH}.  
Then from definition follows that sets of vectors $\{(u_i,0)\}_{i=1}^k\cup\{(0,w_j)\}_{j=1}^s\subset\partial(D\times K)$ and positive numbers $\{c_i\}_{i=1}^k\cup\{d_j\}_{j=1}^s$ satisfying all conditions from theorem \ref{JohnTH} for convex bodies $D\times K$. Therefore the ball $B^{n+m}_2$ is the John ellipsoid. Thus
$$
j(D\times K)=(0,0)=(j(D),j(K)).
$$

General case of the statement follows from affine invariance of L\"owner ellipsoid
%%%%%%%%%%%%%%%%%%%%%%%%%%%%%%%%%%%%%%%%%%%%%%%%
%%                                                           СЛУЧАЙ ЭЛЛИПСОЙДА ЛЕВНЕРА
%%%%%%%%%%%%%%%%%%%%%%%%%%%%%%%%%%%%%%%%%%%%%%%%
\medbreak
Let us turn to the proof of the statement for the center of the L\"owner ellipsoid of the convex body $ D\times K.$

By the same argument as in the statement of center of John ellipsoid it is enough to consider the case then L\"owner ellipsoids of convex bodies are the balls $B^n_2$ and $B^m_2$ respectively. Then from theorem \ref{JohnTH} there exist vectors $\{v_i\}_{i=1}^k\subset\partial D\cap\partial B^n_2$ and $\{z_j\}_{j=1}^s\subset\partial K\cap\partial B^m_2$  and positive numbers $\{a_i\}_{i=1}^k$ and $\{b_j\}_{j=1}^s$ for convex bodies $D$ and $K$ respectively.
Consider the sets of vectors $\{(v_i,z_j)\}_{i=1,j=1}^{k,s}\subset\partial D\times\partial K\subset\partial(D\times K)$ and positive numbers $\{\lambda a_ib_j\}_{i=1,j=1}^{k,s}$, where $\lambda$ is positive number which we define below. So we have
\begin{align}
\label{formula1}
\sum_{i=1}^k\sum_{j=1}^s\lambda a_ib_j(v_i,z_j)=\lambda\left(\sum_{j=1}^sb_j\sum_{i=1}^ka_iv_i,\sum_{i=1}^ka_i\sum_{j=1}^sb_jz_j\right)=0
\end{align}
and
\begin{align}
\label{formula2}
\sum_{i=1}^k\sum_{j=1}^s\lambda a_ib_j\langle(x,y),(v_i,z_j)\rangle^2=\lambda\sum_{j=1}^sb_j\|x\|^2+\lambda\sum_{i=1}^ka_i\|y\|^2
\end{align}
for all $(x,y)\in\mathbb{R}^n\times\mathbb{R}^m$.

Define $(n+m)$-dimensional ellipsoid
$$
\mathcal{E}=\left\{(x,y)\in\mathbb{R}^n\times\mathbb{R}^m : \lambda\|x\|^2\left(\sum_{j=1}^sb_j\right)+\lambda\|y\|^2\left(\sum_{i=1}^ka_i\right)\leq1\right\},
$$
where $\lambda$ is chosen so that $\partial B_2^n\times\partial B^m_2\subset\partial\mathcal{E}$. In other words,
\begin{align*}
\lambda&:=\frac1{\sum_{j=1}^sb_j+\sum_{i=1}^ka_i}.
\end{align*}
We note that $\mathcal{E}$ is affine image of $B^{n+m}_2$ by diagonal matrix
\begin{align*}
T(x,y)=\left(\frac{x}{\sqrt{\lambda\sum_{j=1}^sb_j}},\frac{y}{\sqrt{\lambda\sum_{i=1}^ka_i}}\right).
\end{align*}
From identities \eqref{formula1} and \eqref{formula2} and theorem \ref{JohnTH} follows that the ball $B^{n+m}_2$ is L\"owner ellipsoid of convex body $T(D\times K)$. Since L\"owner ellipsoid is affine invariant, ellipsoid $\mathcal{E}=T^{-1}B^{n+m}_2$ is L\"owner ellipsoid for $D\times K.$ Thus we have
$$
l(D\times K)=(0,0)=(l(D),l(K)).
$$
\medbreak
%%%%%%%%%%%%%%%%%%%%%%%%%%%%%%%%%%%%%%%%%%%%%%%%
%%                                                              СЛУЧАЙ ТОЧКИ САНТАЛО
%%%%%%%%%%%%%%%%%%%%%%%%%%%%%%%%%%%%%%%%%%%%%%%%
To proof the last statement about Santal\'o point we need the next well-known facts \cite{Shcneider}:

{\bf Fact 1.} The interioir point $x$ of convex body $K$ is Santal\'o point if and only if $0$ is centroid of $(K-x)^\circ$.

{\bf Fact 2.} For any $n$-dimensilonal convex bodies $K$ and $L$ the polar of cartesian product $(K\times L)^\circ$ is $\conv((K^\circ,0),(0,L^\circ)).$

It is enough to consider the case $s(K)=0$ and $s(L)=0$. By the fact 1, we need to prove that from conditions for centroids $g(K^\circ)=0$ and $g(L^\circ)=0$ follows the identity $g(\conv((K^\circ,0),(0,L^\circ)))=0$.
Using Fubini's theorem we have
\begin{multline*}
g(\conv((K^\circ,0),(0,L^\circ)))=\qquad\qquad\qquad\qquad\qquad\qquad\\
=\frac1{\vol_{2n}(\conv((K^\circ,0),(0,L^\circ)))}\left(\int_{\conv((K^\circ,0),(0,L^\circ))}z\,dz\,dw,\int_{\conv((K^\circ,0),(0,L^\circ))}w\,dz\,dw\right)=\qquad\\
\quad=\frac1{\vol_{2n}(\conv((K^\circ,0),(0,L^\circ)))}\left(\int_{L^\circ}\int_{(1-\|w\|_{L^\circ})K^\circ}z\,dz\,dw,\int_{K^\circ}\int_{(1-\|z\|_{K^\circ})L^\circ}w\,dw\,dz\right)=(0,0).
\end{multline*}
\end{proof}
\medbreak

\medbreak
\section{Preliminary constructions of convex bodies for measures of symmetry $\varphi_{g,j}$ and $\varphi_{g,l}$}
In this section for the measure $\varphi_{g,j}$ (respectively $\varphi_{g,l}$) we shall consider preliminary convex bodies $F^{(1)}_n$ and $F^{(2)}_n\in\mathcal{K}_n$ (respectively $W^{(1)}_n$ and $W^{(2)}_n\in\mathcal{K}_n$) with some properties.

Centroid of convex body $F^{(1)}_n$ (respectively $W^{(1)}_n$) is "separated" from its boundary with increasing $n$, whereas its center of John ellipsoid (respectively L\"owner ellipsoid) is "close" to the boundary. Centroid and center of John ellipsoid (respectively L\"owner ellipsoid) of convex body  $F^{(2)}_n$ (respectively $W^{(2)}_n$) will have the opposite properties.

Lets consider the convex bodies
\begin{align*}
F^{(1)}_n&=\conv\left(\left(B^n_\infty,0\right),\left(B^n_2,1\right)\right),\\
F^{(2)}_n&=\conv\left(\left(\sqrt{\frac{n}{2e\pi}}B^n_2,0\right),\left(\frac12B^n_\infty,1\right)\right),\\
W^{(1)}_n&=\conv\left(\left(B^n_2,0\right),\left(\frac1{\sqrt{n}}B^n_\infty,1\right)\right)\\
\text{and}\qquad\qquad\qquad\qquad\qquad\qquad&\qquad\qquad\qquad\qquad\qquad\qquad\qquad\qquad\qquad\qquad\qquad\qquad\qquad
\end{align*}
\begin{align*}
W^{(2)}_n&=\conv\left(\left(\frac{n}eB^n_1,0\right),\left(B^n_\infty,1\right)\right).\\
\end{align*}

We note that each of convex bodies above is invariant by nontrivial rotation around an axis in the direction of the vector $e_{n+1}=(0,\ldots,0,1)$. Hence centroids and the centers of L\"owner and John ellipsoids of convex bodies $F^{(1)}_n$, $F^{(2)}_n,$ $W^{(1)}_n$ and $W^{(2)}_n$ have form
\begin{align}
\label{CentroidLJ}
&g(F_n^{(i)})=g^{n+1}(F_n^{(i)})e_{n+1}\mbox{ and }j(F_n^{(i)})=j^{n+1}(F_n^{(i)})e_{n+1},\\
&g(W_n^{(i)})=g^{n+1}(W_n^{(i)})e_{n+1}\mbox{ and } l(W_n^{(i)})=l^{n+1}(W_n^{(i)})e_{n+1},
\end{align}
where $g^{n+1}$, $l^{n+1}$ and $j^{n+1}$ are corresponding $(n+1)$-th coordinates of points for the convex bodies $F_n^{(i)}$ and $W_n^{(i)}$ for $i=1,2.$

Now we ready to prove the next four lemmas.
%%%%%%%%%%%%%%%%%%%%%%%%%%%%%%%%%%%%%%%%%%%%%%%%
%%                                                     ЛЕММА ЦЕТРМАССДЖОН
%%%%%%%%%%%%%%%%%%%%%%%%%%%%%%%%%%%%%%%%%%%%%%%%
\begin{lemma}
\label{lemmaCentroidJohn}
Let $F_n^{(1)}$ be convex bodies defined by
$
F^{(1)}_n=\conv\left(\left(B^n_\infty,0\right),\left(B^n_2,1\right)\right).
$ Then
$$
j(F_n^{(1)})=\frac12e_{n+1}\quad\text{and}\quad g(F_n^{(1)})=g^{n+1}(F_n^{(1)})e_{n+1},
$$
where $g^{n+1}(F_n^{(1)})=\frac1n\left(1+\frac{\pi}2+\frac{\displaystyle e^{-\frac{\pi}{4}}}{\displaystyle\erf\left(\sqrt\pi/2\right)+1}\right)+o\left(\frac1n\right).$ Here $\erf(x)=\frac2{\sqrt\pi}\int_0^xe^{-t^2}\,dt$ is an error function (see e.g. \cite{Bender} \S16.2).
\end{lemma}
\begin{proof}
By remark \eqref{CentroidLJ} it is enough to prove the statement for corresponding $(n+1)$-th coordinates of the points.

We first prove the statement about $j^{n+1}(F_n^{(1)})=\frac12$. Let $J$ be the John ellipsoid of $F_n^{(1)}$. Since John ellipsoid is unique, the ellipsoid $J$ satisfy
$$
J=\left\{(x,t)\in\mathbb{R}^n\times\mathbb{R} : \frac{\|x\|^2}{a^2}+\frac{(t-j^{n+1}(F_n^{(1)}))^2}{b^2}\leq1\right\},
$$
where $a$ and $b$ are positive numbers dependent on $n$.
 For $t\in[0;1]$ the section $J(t)=\{x\in\mathbb{R}^n : (x,t)\in J\}$ is $n$-dimensional euclidean ball contained in John ellipsoid of $F_n^{(1)}$. Since John ellipsoid for $B^n_\infty$ is equal $B^n_2,$ we have inclusions $J(0)\subset B^n_2$ and $J(1)\subset B^n_2$. Consequently,
$$
J\subset\conv\left(\left(B^n_2,0\right),\left(B^n_2,1\right)\right)\subset F_n^{(1)}=\conv\left(\left(B^n_\infty,0\right),\left(B^n_2,1\right)\right).
$$
From maximality of volume of John ellipsoid $J$ follow that $J$ is also Jonh ellipsoid for cylinder $\conv\left(\left(B^n_2,0\right),\left(B^n_2,1\right)\right)$. By the symmetry about point $(0,\ldots,0,\frac12)$ of convex body $\conv\left(\left(B^n_2,0\right),\left(B^n_2,1\right)\right),$ this cylinder has unique affine invariant point. Consequently, this point coincides with the center of ellipsoid $J$. Hence $j^{n+1}(F_n^{(1)})=\frac12$.
\medbreak
Now prove the asymptotic for  $g^{n+1}(F_n^{(1)})$. Applying formula \eqref{CentroidForm} for convex body $F_n^{(1)}$ and computations \eqref{BallVoll} and \eqref{MixVolForm} we get
\begin{align*}
g^{n+1}(F_n^{(1)})&=\frac1{n+2}\frac{\displaystyle\sum_{k=0}^n(k+1)V_{n-k,k}\big(B^n_\infty, B^n_2\big)}{\displaystyle\sum_{k=0}^nV_{n-k,k}\big(B^n_\infty, B^n_2\big)}=\frac1{n+2}\frac{\displaystyle\sum_{k=0}^n(k+1)2^{-k}\vol_k(B^k_2)}{\displaystyle\sum_{k=0}^n2^{-k}\vol_k(B^k_2)}=\\
&=\frac1{n+2}+\frac1{n+2}\frac{\displaystyle\sum_{k=0}^nk\left(\frac{\sqrt{\pi}}{2}\right)^k\frac1{\Gamma\left(\frac{k}2+1\right)}}{\displaystyle\sum_{k=0}^n\left(\frac{\sqrt{\pi}}{2}\right)^k\frac1{\Gamma\left(\frac{k}2+1\right)}}.
\end{align*}
Denote by $f$ the function
$$
f(x)=\sum_{k=0}^\infty\left(\frac{x}2\right)^k\frac1{\Gamma\left(\frac{k}2+1\right)}.
$$
By the equality
$$
\lim_{n\to\infty}\frac{\displaystyle\sum_{k=0}^nk\left(\frac{\sqrt{\pi}}{2}\right)^k\frac1{\Gamma\left(\frac{k}2+1\right)}}{\displaystyle\sum_{k=0}^n\left(\frac{\sqrt{\pi}}{2}\right)^k\frac1{\Gamma\left(\frac{k}2+1\right)}}=\frac{\displaystyle\sqrt{\pi}f'\left(\sqrt{\pi}\right)}{f\left(\sqrt{\pi}\right)},
$$
it is enough to find explicit formula for $f$. We shall seek $f$ as the solution of certain differential equation with initial condition $f(0)=1$. From identities 
$$
\sum_{k=0}^nk\left(\frac{x}2\right)^k\frac1{\Gamma\left(\frac{k}2+1\right)}=x\sum_{k=1}^n\left(\frac{x}2\right)^{k-1}\frac1{\Gamma\left(\frac{k}2\right)}=x\left(\frac1{\sqrt{\pi}}+\frac{x}2\left(\sum_{k=0}^{n-2}\left(\frac{x}2\right)^k\frac1{\Gamma\left(\frac{k}2+1\right)}\right)\right)
$$
we get differential equation for $f$
$$
xf'(x)=x\left(\frac1{\sqrt{\pi}}+\frac{x}2f(x)\right).
$$
Solving this equation, we have 
$$
f(x)=\left(\frac1{\sqrt{\pi}}\int_0^xe^{-\frac{t^2}{4}}\,dt+1\right)e^{\frac{x^2}{4}}=\left(\erf\left(\frac{x}2\right)+1\right)e^{\frac{x^2}4}.
$$
Hence
\begin{align*}
\lim_{n\to\infty}\frac{\displaystyle\sum_{k=0}^nk\left(\frac{\sqrt{\pi}}{2}\right)^k\frac1{\Gamma\left(\frac{k}2+1\right)}}{\displaystyle\sum_{k=0}^n\left(\frac{\sqrt{\pi}}{2}\right)^k\frac1{\Gamma\left(\frac{k}2+1\right)}}=\frac{\displaystyle\sqrt{\pi}f'\left(\sqrt{\pi}\right)}{f\left(\sqrt{\pi}\right)}=\frac{\pi}2+\frac{\displaystyle e^{-\frac{\pi}{4}}}{\displaystyle\erf\left({\sqrt\pi}/2\right)+1}.
\end{align*}
\end{proof}

%%%%%%%%%%%%%%%%%%%%%%%%%%%%%%%%%%%%%%%%%%%%%%%%
%%                                                     ЛЕММАЛЕВНЕРЦЕНТРМАСС
%%%%%%%%%%%%%%%%%%%%%%%%%%%%%%%%%%%%%%%%%%%%%%%%

\begin{lemma}
\label{lemmaJohnCentroid}
Let $F^{(2)}_n$ be convex body defined by 
$F^{(2)}_n=\conv\left(\left(\sqrt{\frac{n}{2e\pi}}B^n_2,0\right),\left(\frac12B^n_\infty,1\right)\right).$ Then
\begin{align*}
j(F^{(2)}_n)=j^{n+1}(F^{(2)}_n)e_{n+1}\quad\text{and}\quad
g(F^{(2)}_n)=g^{n+1}(F^{(2)}_n)e_{n+1}.
\end{align*}
where
$j^{n+1}(F^{(2)}_n)=\frac1n+o\left(\frac1n\right),$ and $g^{n+1}(F_n^{(2)})=1-\frac1e-\frac1n\left(1-\frac2e\right)+o\left(\frac1n\right).$
\end{lemma}
\begin{proof}
\medbreak
Let $J$ be the John ellipsoid of convex body $F_n^{(2)}$. From its uniqueness follows $J$ has form
\begin{align}
\label{elDec2}
\mathcal{E}_{a,b,c}=\left\{(x,t)\in\mathbb{R}^{n+1}=\mathbb{R}^n\times\mathbb{R}: \frac{\|x\|^2}{a^2}+\frac{(t-c)^2}{b^2}\leq1\right\}
\end{align}
for some $a>0,  b>0$ and $c\in[0 ;1]$.

Denote $\sqrt{\frac{n}{2e\pi}}$ by $r_n.$ Note that John ellipsoids of convex bodies $r_nB^n_1$ and $\frac12B^n_\infty$ respectively equal $r_nB^n_2$ and $\frac12B^n_2$. For any $t\in[0;1]$ from lemma \ref{wellknown} follows $J(t)\subset r_n(1-t)B^n_2+\frac12tB^n_2=\left(r_n-t(r_n-\frac12)\right)B^n_2$. Therefore parameters $a,b$ and $c$ satisfy inequality
$$
a^2\left(1-\frac{(t-c)^2}{b^2}\right)\leq\left(r_n-t(r_n-\frac12)\right)^2
$$ 
for any $t\in[0;1].$ Since John ellipsoid has maximal volume, exists $t\in[0;1]$ for which this inequality turns into equality. This condition is equivalent to equality
$$
a^2=\left(r_n-c\left(r_n-\frac12\right)\right)^2-\left(r_n-\frac12\right)^2b^2.
$$

Consider the function of square of the volume 
$$
f(b,c)=\vol_{n+1}^2(\mathcal{E}_{a,b,c})=b^2a^{2n}=b^2\left(\left(r_n+c\left(r_n-\frac12\right)\right)^2-\left(r_n-\frac12\right)^2b^2\right)^n.
$$
John ellipsoid $J$ maximize this function. Since $f(b,c)$ is decreasing by parameter $c$ on $[0;1]$ with restriction $b\leq\min\{c,1-c\}$, maximum is attained in $b=c$. Thus $j^{n+1}(W_n^{(2)})$ maximize the function $f(c,c)=c^2\left(r_n^2-2r_n(r_n-\frac12)c\right)^n$ where $c\in[0;1]$. Hence we have
$$
j^{n+1}(F_n^{(2)})=\frac{r_n}{(r_n-\frac12)(n+2)}=\frac{\sqrt{\frac{n}{2e\pi}}}{(\sqrt{\frac{n}{2e\pi}}-\frac12)(n+2)}=\frac1n+o\left(\frac1n\right).
$$
\medbreak

The statement about centroid follows from the same computations as in Appendix A from~\cite{MMCSEW}. These arguments are given here for completeness.

 From lemma \ref{LemmaMixedVol} follows
\begin{align}
g^{n+1}(F^{(2)}_n)&=\frac{1}{n+2}\frac{\sum_{k=0}^n(k+1)V_{n-k,k}\big(\sqrt{\frac{n}{2e\pi}}B^n_2, \frac12B^n_\infty\big)}{\sum_{k=0}^nV_{n-k,k}\big(\sqrt{\frac{n}{2e\pi}}B^n_2, \frac12B^n_\infty\big)}=\notag\\
&=\frac{1}{n+2}\frac{\displaystyle\sum_{k=0}^n(k+1)\frac1{\left(\frac{n}{2e}\right)^\frac{k}2\Gamma\left(1+\frac{n-k}2\right)}}{\displaystyle\sum_{k=0}^n\frac1{\left(\frac{n}{2e}\right)^\frac{k}2\Gamma\left(1+\frac{n-k}2\right)}}=\notag\\
\label{CentrW_n}
&=\frac{n+1}{n+2}-\frac{1}{n+2}\frac{\displaystyle\sum_{k=0}^n\frac{k\left(\frac{n}{2e}\right)^\frac{k}{2}}{\Gamma\left(1+\frac{k}2\right)}}{\displaystyle\sum_{k=0}^n\frac{\left(\frac{n}{2e}\right)^\frac{k}2}{\Gamma\left(1+\frac{k}2\right)}}.
\end{align}
For every $x\geq0,$
\begin{multline*}
\sum_{k=0}^n\frac{kx^k}{\Gamma\left(1+\frac{k}2\right)}=\sum_{k=1}^n\frac{kx^k}{\frac{k}2\Gamma\left(\frac{k}2\right)}=2x\sum_{k=1}^n\frac{x^{k-1}}{\Gamma\left(\frac{k}2\right)}=\\
=2x\left(\frac1{\Gamma\left(\frac12\right)}+\sum_{k=2}^n\frac{x^{k-1}}{\Gamma\left(\frac{k}2\right)}\right)=2x\left(\frac1{\Gamma\left(\frac12\right)}+x\sum_{k=0}^{n-2}\frac{x^{k}}{\Gamma\left(1+\frac{k}2\right)}\right).
\end{multline*}
Substituting $x=\sqrt{\frac{n}{2e}}$ and denoting by
$$
A_n=\sum_{k=0}^n\frac{k\left(\frac{n}{2e}\right)^\frac{k}{2}}{\Gamma\left(1+\frac{k}2\right)}\quad\text{and}\quad B_n= \sum_{k=0}^n\frac{\left(\frac{n}{2e}\right)^\frac{k}2}{\Gamma\left(1+\frac{k}2\right)}
$$
it follows
\begin{multline*}
A_n=2\sqrt{\frac{n}{2e}}\left(\frac1{\Gamma\left(\frac12\right)}+\sqrt{\frac{n}{2e}}\sum_{k=0}^{n-2}\frac{\left(\frac{n}{2e}\right)^{\frac{k}2}}{\Gamma\left(1+\frac{k}2\right)}\right)=\\
=2\sqrt{\frac{n}{2e}}\left(\frac1{\Gamma\left(\frac12\right)}+\sqrt{\frac{n}{2e}}\left(B_n-\frac{\left(\frac{n}{2e}\right)^\frac{n-1}2}{\Gamma(1+\frac{n-1}2)}-\frac{\left(\frac{n}{2e}\right)^\frac{n}2}{\Gamma(1+\frac{n}2)}\right)\right).
\end{multline*}
Consequently,
\begin{multline*}
g^{n+1}(F^{(2)}_n)=\frac{n+1}{n+2}-\frac1{n+2}\frac{A_n}{B_n}=\\
=\frac{n+1}{n+2}-\frac{2\sqrt{\frac{n}{2e}}}{n+2}\left(\frac1{\Gamma\left(\frac12\right)B_n}+\sqrt{\frac{n}{2e}}\left(1-\frac{\frac{\left(\frac{n}{2e}\right)^\frac{n-1}2}{\Gamma(1+\frac{n-1}2)}+\frac{\left(\frac{n}{2e}\right)^\frac{n}2}{\Gamma(1+\frac{n}2)}}{B_n}\right)\right).
\end{multline*}
Taking into account that
$$
B_n\geq\frac{\left(\frac{n}{2e}\right)^2}{\Gamma(3)}
$$
for $n\geq4,$ we get
$$
g^{n+1}(F^{(2)}_n)=1-\frac1e-\frac1n\left(1-\frac2e\right)+o\left(\frac1n\right).
$$
\end{proof}

%%%%%%%%%%%%%%%%%%%%%%%%%%%%%%%%%%%%%%%%%%%%%%%%
%%                                                      ЛЕММА ЦЕТРМАССЛЕВНЕР
%%%%%%%%%%%%%%%%%%%%%%%%%%%%%%%%%%%%%%%%%%%%%%%%
\begin{lemma}
\label{lemmaCentroidLowner}
Let $W^{(1)}_n$ be convex bodies defined by 
$
W^{(1)}_n=\conv\left(\left(B^n_2,0\right),\left(\frac1{\sqrt{n}}B^n_\infty,1\right)\right).
$
Then
\begin{align*}
l(W^{(1)}_n)=\frac12e_{n+1}\quad\text{and}\quad
g(W^{(1)}_n)=g^{n+1}(W^{(1)}_n)e_{n+1},
\end{align*}
where $g^{n+1}(W^{(1)}_n)=\frac1n\frac1{1-\sqrt{\frac2\pi}}+o\left(\frac1n\right).$
\end{lemma}
\begin{proof}
We shall show that $l^{n+1}(W^{(1)}_n)=\frac12$. Denote by $L$ the L\"owner ellipsoid of convex body $W^{(1)}_n$. Since L\"ovner ellipsoid is unique, ellipsoid $L$ has the form
$$
L=\left\{(x,t)\in\mathbb{R}^n\times\mathbb{R} : \frac{\|x\|^2}{a^2}+\frac{(t-l^{n+1}(W^{(1)}_n))^2}{b^2}\leq1\right\},
$$
where $a$ and $b$ are some positive numbers dependent on $n$.
For $t\in[0;1]$ section $L(t)=\{x\in\mathbb{R}^n : (x,t)\in L\}$ is $n$-dimensional  euclidian ball containing L\"owner ellipsoid of the section $W_n^{(1)}(t)$. Since L\"owner ellipsoid for $\frac1{\sqrt{n}}B^n_\infty$ is the ball $B^n_2$ we have the inclusions $B^n_2\subset L(0)$ and $B^n_2\subset L(1)$. Consequently,
$$
W^{(1)}_n=\conv\left(\left(B^n_2,0\right),\left(\frac1{\sqrt{n}}B^n_\infty,1\right)\right)\subset\conv\left(\left(B^n_2,0\right),\left(B^n_2,1\right)\right)\subset L.
$$
From minimality of L\"owner ellipsoid $L$ follows that $L$ is also L\"owner ellipsoid for cylinder $\conv\left(\left(B^n_2,0\right),\left(B^n_2,1\right)\right)$. Since this cylinder is symmetric about the point $(0,\ldots,0,\frac12),$ this body has unique affine invariant point. Therefore, this point is also center of ellipsoid $L$. Hence $l^{n+1}(W^{(1)}_n)=\frac12$.
\medbreak
Find the asymptotic for $g^{n+1}(W^{(1)}_n)$. From lemma \ref{LemmaMixedVol} follows
\begin{align*}
g^{n+1}(W^{(1)}_n)&=\frac{1}{n+2}\frac{\sum_{k=0}^n(k+1)V_{n-k,k}\big(B^n_2, \frac1{\sqrt{n}}B^n_\infty\big)}{\sum_{k=0}^nV_{n-k,k}\big(B^n_2, \frac1{\sqrt{n}}B^n_\infty\big)}=\qquad\qquad\\
&=\frac{1}{n+2}\frac{\displaystyle\sum_{k=0}^n(k+1)\left(\frac2{\sqrt{n}}\right)^k\frac{\pi^{\frac{n-k}2}}{\Gamma\left(1+\frac{n-k}2\right)}}{\displaystyle\sum_{k=0}^n\left(\frac2{\sqrt{n}}\right)^k\frac{\pi^{\frac{n-k}2}}{\Gamma\left(1+\frac{n-k}2\right)}}=
\end{align*}
\begin{align*}
\qquad\quad=\frac1{n+2}+\frac{1}{n+2}\frac{\displaystyle\sum_{k=0}^nk\left(\frac{\pi n}4\right)^\frac{n-k}2\frac{1}{\Gamma\left(1+\frac{n-k}2\right)}}{\displaystyle\sum_{k=0}^n\left(\frac{\pi n}4\right)^\frac{n-k}2\frac{1}{\Gamma\left(1+\frac{n-k}2\right)}}.
\end{align*}
 Now we shall compute asymptotic for sums 
\begin{align}
\label{Den43}
\sum_{k=0}^n\left(\frac{\pi n}4\right)^\frac{n-k}2\frac{1}{\Gamma\left(1+\frac{n-k}2\right)}
\end{align}
and
\begin{align}
\label{Num43}
\sum_{k=0}^nk\left(\frac{\pi n}4\right)^\frac{n-k}2\frac{1}{\Gamma\left(1+\frac{n-k}2\right)}.
\end{align}
Each term of both sums are less than
$$
\left(\frac{\pi n}4\right)^\frac{n}2\frac{1}{\Gamma\left(1+\frac{n}2\right)}\sim\frac1{\sqrt{\pi n}}\left(\frac{\pi e}{2}\right)^\frac{n}2.
$$
The sum \eqref{Den43} can be presented in the next form
\begin{align*}
\sum_{k=0}^n\left(\frac{\pi n}4\right)^\frac{n-k}2\frac{1}{\Gamma\left(1+\frac{n-k}2\right)}=\sum_{0\leq k<n^{1/4}}\left(\frac{\pi n}4\right)^\frac{n-k}2\frac{1}{\Gamma\left(1+\frac{n-k}2\right)}+\sum_{n^{1/4}\leq k\leq n}\left(\frac{\pi n}4\right)^\frac{n-k}2\frac{1}{\Gamma\left(1+\frac{n-k}2\right)}.
\end{align*}
The second sum is $o\left(n^{-\frac32}\left(\frac{\pi e}{2}\right)^\frac{n}2\right)$ since
\begin{multline*}
\sum_{n^{1/4}\leq k\leq n}\left(\frac{\pi n}4\right)^\frac{n-k}2\frac{1}{\Gamma\left(1+\frac{n-k}2\right)}\leq n\left(\frac{\pi n}4\right)^\frac{n-n^{1/4}}2\frac{1}{\Gamma\left(1+\frac{n-n^{1/4}}2\right)}\leq\\ 
\leq Cn\left(\frac{\pi n}4\right)^\frac{n-n^{1/4}}2e^\frac{n-n^{1/4}}2\left(\frac{n-n^{1/4}}2\right)^{-\frac{n-n^{1/4}+1}2}\leq Cn^{-3/2}\left(\frac{\pi e}{2}\right)^\frac{n}2n^2\left(\frac2\pi\right)^\frac{n^{1/4}}2.
\end{multline*}

To compute asymptotic of $\sum_{0\leq k\leq n^{1/4}}\ldots$ we shall use well-known precision of Stirling's formula (\cite{Bender} \S12.33)
$
\Gamma(x+1)=x^xe^{-x}\sqrt{2\pi x}e^{\theta(x)},
$
where $0\leq\theta(x)\leq\frac1x$. Remark that from this precision for $n-n^{1/4}<x<n$ follows equivalence
$
\Gamma(x+1)=x^xe^{-x}\sqrt{2\pi x}(1+o(1)),
$
where $o(1)$ is uniformly bounded by $x$. Hence for $0\leq k\leq n^{1/4}$ we get
\begin{multline*}
\left(\frac{\pi n}4\right)^\frac{n-k}2\frac{1}{\Gamma\left(1+\frac{n-k}2\right)}=\left(\frac{\pi n}4\right)^\frac{n-k}2\left(\frac{n-k}2\right)^{-\frac{n-k}2}e^{\frac{n-k}2}\frac{(1+o(1))}{\sqrt{\pi (n-k)}}=\\
=\frac1{\sqrt{\pi n}}\left(\frac{\pi e}2\right)^\frac{n}2\left(\frac2\pi\right)^\frac{k}2(1+o(1)).
\end{multline*}
Consequently,
\begin{multline*}
\sum_{k=0}^{n^{1/4}}\left(\frac{\pi n}4\right)^\frac{n-k}2\frac{1}{\Gamma\left(1+\frac{n-k}2\right)}=\frac1{\sqrt{\pi n}}\left(\frac{\pi e}2\right)^\frac{n}2\sum_{k=0}^{n^{1/4}}\left(\frac{2}{\pi}\right)^\frac{k}2\left(1+o(1)\right)=\\
=\frac1{\sqrt{\pi n}}\left(\frac{\pi e}2\right)^\frac{n}2\left(\frac1{1-\sqrt{\frac2\pi}}+o(1)\right).
\end{multline*}
The asymptotic for \eqref{Num43} is computed in the similar way
\begin{multline*}
\sum_{k=0}^{n^{1/4}}k\left(\frac{\pi n}4\right)^\frac{n-k}2\frac{1}{\Gamma\left(1+\frac{n-k}2\right)}=\frac1{\sqrt{\pi n}}\left(\frac{\pi e}2\right)^\frac{n}2\sum_{k=0}^{n^{1/4}}\left(\frac{2}{\pi}\right)^\frac{k}2\left(k+o(k)\right)=\\
=\frac1{\sqrt{\pi n}}\left(\frac{\pi e}2\right)^\frac{n}2\left(\frac{\sqrt{\frac2\pi}}{\left(1-\sqrt{\frac2\pi}\right)^2}+o(1)\right).
\end{multline*}
Hence we have
\begin{multline*}
g^{n+1}(W^{(1)}_n)=\frac1{n+2}+\frac{1}{n+2}\frac{\displaystyle\sum_{k=0}^nk\left(\frac{\pi n}4\right)^\frac{n-k}2\frac{1}{\Gamma\left(1+\frac{n-k}2\right)}}{\displaystyle\sum_{k=0}^n\left(\frac{\pi n}4\right)^\frac{n-k}2\frac{1}{\Gamma\left(1+\frac{n-k}2\right)}}=\\
=\left(\frac1n+o\left(\frac1n\right)\right)\left(1+\frac{\sqrt{\frac2\pi}}{1-\sqrt{\frac2\pi}}+o\left(1\right)\right)=\frac1n\frac1{1-\sqrt{\frac2\pi}}+o\left(\frac1n\right).
\end{multline*}
\end{proof}

%%%%%%%%%%%%%%%%%%%%%%%%%%%%%%%%%%%%%%%%%%%%%%%%
%%                                                     ЛЕММАЛЕВНЕРЦЕНТРМАСС
%%%%%%%%%%%%%%%%%%%%%%%%%%%%%%%%%%%%%%%%%%%%%%%%

\begin{lemma}
\label{lemmaLownerCentroid}
Let $W^{(2)}_n$ be convex body, defined by $W^{(2)}_n=\conv\left(\left(\frac{n}eB^n_1,0\right),\left(B^n_\infty,1\right)\right).$ Then
\begin{align*}
l(W^{(2)}_n)=l^{n+1}(W^{(2)}_n)e_{n+1}\quad\text{and}\quad
g(W^{(2)}_n)=g^{n+1}(W^{(2)}_n)e_{n+1}.
\end{align*}
where
$l^{n+1}(W^{(2)}_n)=\frac1n+o\left(\frac1n\right),$ and $g^{n+1}(W_n^{(2)})=1-\frac1e-\frac1n\left(1-\frac2e\right)+o\left(\frac1n\right)$.
\end{lemma}
\begin{proof}

Let be $L$ L\"owner ellipsoid of convex body $W^{(2)}_n$. From uniqueness of L\"owner ellipsoid follows that $L$ has form
\begin{align}
\label{elDec2}
\mathcal{E}_{a,b,c}=\left\{(x,t)\in\mathbb{R}^{n+1}=\mathbb{R}^n\times\mathbb{R}: \frac{\|x\|^2}{a^2}+\frac{(t-c)^2}{b^2}\leq1\right\}
\end{align}
for some $a>0,  b>0$ and $c\in[0 ;1]$.

Denote $\frac{n}e$ by $r_n.$
From definition of L\"owner ellipsoid follows the inclusion $W_n^{(2)}(t)= (1-t)B^n_{\infty}+tr_nB^n_1\subset L(t)$ for any $t\in[0;1].$ On the other hand, for any ellipsoid $\mathcal{E}_{a,b,c}$ from inclusions $B^n_{\infty}\subset L(0)$ and $B^n_1\subset L(1)$ by convexity we get inclusions $W_n^{(2)}(t)= (1-t)B^n_{\infty}+tr_nB^n_1\subset \mathcal{E}_{a,b,c}(t)$ for any $t\in[0;1].$ Note that L\"owner ellipsoid for convex bodies $r_nB^n_1$ and $B^n_\infty$ 
are $r_nB^n_2$ and $\sqrt{n}B^n_2$. The sections $L(t)=\{x\in\mathbb{R}^n : (x,t)\in L\}$ of ellipsoid $L$ are $n$-dimensional euclidian balls with some scaling. So from minimality of L\"owner ellipsoid follows $L(0)=r_nB^n_2$, $L(1)=\sqrt{n}B^n_2$ and the next conditions for parameters $a,b$ and $c$
\begin{align*}
r_n^2=a^2\left(1-\frac{c^2}{b^2}\right)\quad
\text{and}\quad
n=a^2\left(1-\frac{(1-c)^2}{b^2}\right).
\end{align*}
Hence
\begin{align*}
b^2=\frac{r_n^2(1-c)^2-nc^2}{r_n^2-n}\quad
\text{and}\quad
a^2=\frac{r_n^2\left(1-c\right)^2-nc^2}{1-2c}.
\end{align*}
From the positivity of the left-hand sides follows necessary restriction $c\in(0;\frac12).$ 

We shall compute $l^{n+1}(W^{(2)}_n)$, as the minimum point of the square function of the volume
\begin{align*}
f(c)=\vol_{n+1}^2(\mathcal{E}_{a,b,c})&=b^2a^{2n}\vol_{n+1}^2(B^{n+1}_2)=\\
&=\frac{r_n^2(1-c)^2-nc^2}{r_n^2-n}\left(\frac{r_n^2\left(1-c\right)^2-nc^2}{1-2c}\right)^n\vol_{n+1}^2(B^{n+1}_2)
\end{align*}
for $c\in\left(0;\frac12\right)$.

Roots of the equation $f'(c)=0$ are
\begin{align*}
\beta_n^\pm=\frac{(n+3)r_n^2-(n+1)n\pm\sqrt{\big((n+3)r_n^2-(n+1)n\big)^2-4(n+2)\big(r_n^2-n\big)r^2_n}}{2(n+2)\big(r_n^2-n\big)}.
\end{align*}

Since $r_n=\frac{n}e$, for large enough $n$ the root $\beta_n^+$ do not belongs to the interval $(0;\frac12)$.
Consequently,
\begin{align*}
l^{n+1}(W^{(2)}_n)=\beta_n^-\frac{1+\frac3n-\frac{n}{r^2_n}+o\big(\frac1n\big)-\sqrt{1 +\frac2n-\frac{2n}{r^2_n}+o\big(\frac1n\big)}}{2\big(1+\frac2n\big)\big(1-\frac{n}{r^2_n}\big)}=\frac1n+o\left(\frac1n\right).
\end{align*}
\medbreak
Let us turn to the proof of the statement about centroid.
Applying formula \eqref{CentroidForm}  to convex body $W_n^{(2)}$ and using asymptotic \eqref{volBn1} we get
\begin{align*}
g^{n+1}(W_n^{(2)})&=\frac1{n+2}\frac{\displaystyle\sum_{k=0}^n(k+1)2^k\left(\frac{n}e\right)^{n-k}\vol_{n-k}(B^{n-k}_1)}{\displaystyle\sum_{k=0}^n2^k\left(\frac{n}e\right)^{n-k}\vol_{n-k}(B^{n-k}_1)}=\frac1{n+2}\frac{\displaystyle\sum_{k=0}^n\frac{n^{n-k}(k+1)}{e^{n-k}(n-k)!}}{\displaystyle\sum_{k=0}^n\frac{n^{n-k}}{e^{n-k}(n-k)!}}=\\
&=\frac1{n+2}\frac{\displaystyle\sum_{k=0}^n\frac{n^{k}(n-k+1)}{e^{k}k!}}{\displaystyle\sum_{k=0}^n\frac{n^{k}}{e^{k}k!}}= \frac{n+1}{n+2}-\frac1{n+2}\frac{\displaystyle\sum_{k=0}^nk\frac{n^k}{e^kk!}}{\displaystyle\sum_{k=0}^n\frac{n^k}{e^kk!}}.
\end{align*}
Therefore, it suffices to prove the asymptotic
\begin{align}
\label{lem65}
\frac1{n+2}\frac{\displaystyle\sum_{k=0}^nk\frac{n^k}{e^kk!}}{\displaystyle\sum_{k=0}^n\frac{n^k}{e^kk!}}=\frac1e-\frac2{ne}+o\left(\frac1n\right).
\end{align}
Note for any $x\in\mathbb{R}$
\begin{align*}
\sum_{k=0}^n\frac{kx^k}{k!}=\sum_{k=1}^n\frac{kx^k}{k(k-1)!}=x\sum_{k=1}^n\frac{x^{k-1}}{(k-1)!}=x\sum_{k=0}^{n-1}\frac{x^k}{k!}.
\end{align*}
Substituting $x=\frac{n}e$ and dividing both parts by $\sum_{k=0}^n\frac{n^k}{e^kk!}$, we have
\begin{align*}
\frac{\displaystyle\sum_{k=0}^n\frac{kn^k}{e^kk!}}{\displaystyle\sum_{k=0}^n\frac{n^k}{e^kk!}}=\frac{\displaystyle \frac{n}e\sum_{k=0}^{n-1}\frac{n^k}{e^kk!}}{\displaystyle\sum_{k=0}^n\frac{n^k}{e^kk!}}=\frac{n}e\left(1-\frac{\displaystyle\frac{n^n}{e^nn!}}{\displaystyle\sum_{k=0}^n\frac{n^k}{\displaystyle e^kk!}}\right)=\frac{n}{e}\left(1+o\left(\frac1n\right)\right),
\end{align*}
as by Stirling's formula numerator less than $1$ and denominator is more then $\frac{n^2}{2e^2}$. This concludes \eqref{lem65}.
\end{proof}
\medbreak
%%%%%%%%%%%%%%%%%%%%%%%%%%%%%%%%%%%%%%%%%%%%%%%%
%%                                                            РАЗДЕЛ ДЖОН-ЛЕВНЕР
%%%%%%%%%%%%%%%%%%%%%%%%%%%%%%%%%%%%%%%%%%%%%%%%
\section{Two preliminary constructions of convex bodies for measures of symmetry $\varphi_{j,l}$}
As in the previous  section we provide convex bodies $M^{(1)}_n$ and $M^{(2)}_n\in\mathcal{K}_n$ with the next properties: center of John ellipsoid of convex body $M^{(1)}_n$ is "close" to the boundary, whereas center of L\"owner ellipsoid is "separated" from the boundary; for $M^{(2)}_n$ corresponding centers have the opposite property.

Construction of $M^{(1)}_n$ is provided from \cite{MMCSEW}.
%%%%%%%%%%%%%%%%%%%%%%%%%%%%%%%%%%%%%%%%%%%%%%%%
%%                                                        Лемма-конструкция: точка Джона близка к границе
%%%%%%%%%%%%%%%%%%%%%%%%%%%%%%%%%%%%%%%%%%%%%%%%
\begin{lemma}[M. Meyer, C. Schütt, E. M. Werner \cite{MMCSEW}]
\label{JohnLowner}
For convex body $$M^{(1)}_n=\conv((B^n_2,0),(\Delta_n,1))$$ we have
\begin{align*}
l(M^{(1)}_n)=\frac12e_{n+1}\quad\text{and}\quad
j(M^{(1)}_n)=j^{n+1}(M^{(1)}_n)e_{n+1},
\end{align*}
where $j^{n+1}(M^{(1)}_n)=\frac1n+o\left(\frac1n\right)$.
\end{lemma}
%%%%%%%%%%%%%%%%%%%%%%%%%%%%%%%%%%%%%%%%%%%%%%%%
%%                                                        Лемма-конструкция: точка Левнера близка к границе
%%%%%%%%%%%%%%%%%%%%%%%%%%%%%%%%%%%%%%%%%%%%%%%%

By the same argument as in \cite{MMCSEW} the second convex body $M^{(2)}_n$ can be constructed in the following way. 
\begin{lemma}
\label{LownerJohn}
For convex body $M^{(2)}_n=\conv\left((\Delta_n,0),\left(\frac1nB_2^n,1\right)\right)$ we have
\begin{align*}
j(M^{(2)}_n)=\frac12e_{n+1}\quad\text{and}\quad
l(M^{(2)}_n)=l^{n+1}(M^{(2)}_n)e_{n+1},
\end{align*}
where $l^{n+1}(M^{(2)}_n)=\frac1n+o\left(\frac1n\right)$.
\end{lemma}
\begin{proof}
%%%%%%%%%%%%%%%%%%%%%%%%%%%%%%%%%%%%%%%%%%%%%%%%
%%                                                       ЭЛЛИПСОИД ДЖОНА
%%%%%%%%%%%%%%%%%%%%%%%%%%%%%%%%%%%%%%%%%%%%%%%%
Firstly, we prove the statement about the center of John ellipsoid $j(M^{(2)}_n).$ Denote it by $J$.

We note convex body $M^{(2)}_n$ is invariant by some rotations around the axis directed along the vector $e_n=(0,\ldots,0,1)$. Since John ellipsoid is unique, $J$ has the form
\begin{align}
\label{elDec}
\mathcal{E}_{a,b,c}=\left\{(x,t)\in\mathbb{R}^{n+1}=\mathbb{R}^n\times\mathbb{R}: \frac{\|x\|^2}{a^2}+\frac{(t-c)^2}{b^2}\leq1\right\}
\end{align}
for some $a>0,  b>0$ and $c\in[0 ;1]$.

For $t\in[0,1]$ by lemma \ref{wellknown} John ellipsoid for sections $M^{(2)}_n(t)=(1-t)\Delta_n+t\frac1nB^n_2$ is equal 
$((1-t)\frac1n+t\frac1n)B_2^n$. From representation \eqref{elDec}  follows inclusions $J(t)\subset\frac1nB_2^n$. Thus we have necessary conditions for parameters $a,b,c$
\begin{align}
\label{conditionABC}
a\sqrt{1-\frac{(t-c)^2}{b^2}}\leq\frac1n.
\end{align}
valid for all $t\in[0;1].$
Since John ellipsoid has the maximum volume, exists $t\in[0;1]$ for which inequality \eqref{conditionABC} turns into equality.  It is easy to check, that identity attained in $t = c$. Hence $a=\frac1n$.  From inscription of John ellipsoid into $M^{(2)}_n$ follows the codition $b\leq \min(c,1-c)$. As $\vol_{n+1}\left(\mathcal{E}_{\frac1n,b,c}\right)=b\frac1{n^n}\vol_{n+1}\left(B_2^{n+1}\right),$ volume of ellipsoid $\mathcal{E}_{\frac1n,b,c}$ is maximized in $b=c=\frac12$.

On the other hand, ellipsoid $\mathcal{E}_{\frac1n,\frac12,\frac12}$ is contained in convex body $M^{(2)}_n,$ since
$$
\mathcal{E}_{\frac1n,\frac12,\frac12}\subset\left(\frac1nB^n_2\right)\times[0;1]\subset M^{(2)}_n.
$$
Consequently, by uniqueness of John ellipsoid, we have $ J=\mathcal{E}_{\frac1n,\frac12,\frac12}$ and $j^{n+1}(M^{(2)}_n)=\frac12.$
\medbreak
%%%%%%%%%%%%%%%%%%%%%%%%%%%%%%%%%%%%%%%%%%%%%%%%
%%                                                       ЭЛЛИПСОИД ЛЕВНЕРА
%%%%%%%%%%%%%%%%%%%%%%%%%%%%%%%%%%%%%%%%%%%%%%%%
Lets turn to prove statement about center of L\"owner ellipsoid $L$ of convex body $M^{(2)}_n.$ By the same argument above the L\"owner ellipsoid has form \eqref{elDec}.

 From definition of $L$ follows inclusions for sections $M^{(2)}_n(t)=(1-t)\Delta_n+t\frac1nB^n_2\subset L(t)$ for $t\in[0;1]$. On the other hand, for any ellipsoid $\mathcal{E}_{a,b,c}$ from inclusions $\Delta_n\subset\mathcal{E}_{a,b,c}(0)$ and $\frac1nB^n_2\subset\mathcal{E}_{a,b,c}(1)$, by convexity of $M^{(2)}_n$, we have inclusion $M^{(2)}_n(t)=(1-t)\Delta_n+t\frac1nB^n_2\subset \mathcal{E}_{a,b,c}(t)$ for $t\in[0;1]$. Using this remark and minimality of L\"owner ellipsoid we get necessary conditions for parameters  $a,b$ and $c$
\begin{align*}
\frac1{n^2}=a^2\left(1-\frac{(1-c)^2}{b^2}\right)\quad\text{and}\quad1=a^2\left(1-\frac{c^2}{b^2}\right).
\end{align*}
Therefore
\begin{align*}
b^2=\frac{\left(1-\frac1{n^2}\right)c^2-2c+1}{1-\frac1{n^2}}\quad\text{and}\quad a^2=\frac{\left(1-\frac1{n^2}\right)c^2-2c+1}{1-2c}.
\end{align*}
Since the left-hand side is positive, we have
$c<\frac12.$

We shall find $l_{n+1}(M^{(2)}_n)$, as the solution of minimization problem for square volume
\begin{align*}
f(c)=\vol_{n+1}^2(\mathcal{E}_{a,b,c})&=b^2a^{2n}\vol_{n+1}^2(B^{n+1}_2)=\\
&=\frac{\left(1-\frac1{n^2}\right)c^2 -2c+1}{1-\frac1{n^2}}\left(\frac{\left(1-\frac1{n^2}\right)c^2-2c+1}{1-2c}\right)^n\vol_{n+1}^2(B^{n+1}_2)
\end{align*}
where $c\in\left(0;\frac12\right)$.

The roots of $f'(c)=0$ are 
\begin{align*}
\alpha^{\pm}_n=\frac{n^3+3n^2-n-1\pm\sqrt{(n^3+3n^2-n-1)^2-4n^2(n^2-1)(n+2)}}{2(n^2-1)(n+2)}.
\end{align*}
For large enough $n$ we have $\alpha^{+}_n\not\in\left(0;\frac12\right)$.
From uniqueness of L\"owner ellipsoid it follows $l_{n+1}(M^{(2)}_n)=\alpha^-_n$. Thus,
\begin{align*}
l_{n+1}(M^{(2)}_n)=\alpha^-_n&=\frac{1+\frac3n+o\left(\frac1n\right) -\sqrt{1+\frac2{n}+o\left(\frac1n\right)}}{2(1-\frac1{n^2})(1+\frac2n)}=\frac1n+o\left(\frac1n\right).
\end{align*}
\end{proof}
%%КОНСТРУКЦИЯ
\medbreak
\section{Asymptotically sharpness of upper bounds for measures of symmetry}
In this section will be provided the approach of constructing the extremal convex bodies for considered  measures of symmetry.

We need the next geometrical remark, following from direct computations.
\begin{lemma}
\label{fool}
For two different interior points $(x^1,y^1),(x^2,y^2)\in[0;1]\times[0;1]$ denote by $s$ the line passing through this two points Then
$$
\frac{\|(x^1,y^1)-(x^2,y^2)\|}{\vol_1(s\cap[0;1]\times[0;1])}=\frac{|x^2-x^1|\,|y^2-y^1|}{|x^1y^2-x^2y^1|}.
$$
\end{lemma}
We shall illustrate the approach on the measure of symmetry $\varphi_{j,l}.$
\begin{theorem}
\label{MainJL}
Exist convex bodies $M_n\in\mathcal{K}_n$ which satisfy
$$
d(j(M_n),l(M_n))=1-\frac8n+o\left(\frac1n\right).
$$
In the other words, the estimate \eqref{MordhorstJL} is asymptotically sharp up to order $n$.
\end{theorem}
\begin{proof}
Prove this statement for $n=2k$. Consider the next $2k$-dimensional convex bodies
$$
M_{2k}=M_{k-1}^{(1)}\times M_{k-1}^{(2)},
$$
where $M_{k-1}^{(1)}$ and $M_{k-1}^{(2)}$ defined in lemmas \ref{JohnLowner} and \ref{LownerJohn}.
Applying lemma \ref{BodyProd} for convex bodies $M_{2k}$ and using lemmas  \ref{JohnLowner} and \ref{LownerJohn}, we have the next representations for centers of John ellipsoid and L\"owner ellipsoid
\begin{align*}
j(M_{2k})&=(j(M_{k-1}^{(1)}), j(M_{k-1}^{(2)}))=\Bigl(\underbrace{0, 0, 0,\ldots,0,}_{\text{$k-1$ coordinates}}j^{k},\underbrace{0, 0, 0,\ldots,0,}_{\text{$k-1$ coordinates}}\frac12\Bigr),\\
l(M_{2k})&=(l(M_{k-1}^{(1)}), l(M_{k-1}^{(2)}))=\Bigl(\underbrace{0, 0, 0,\ldots,0,}_{\text{$k-1$ coordinates}}\frac12,\underbrace{0, 0, 0,\ldots,0,}_{\text{$k-1$ coordinates}}l^{2k}\Bigr),
\end{align*}
where $j^k=\frac1k+o\left(\frac1k\right)$ and $l^{2k}=\frac1k+o\left(\frac1k\right).$

Denote by $s$ the line passing through the points $j(M_{2k})$ and $l(M_{2k})$. Note that due to representations for $j(M_{2k})$ and $l(M_{2k})$, it is enough to compute the length of intersection  $s$  with the body $M_{2k}$ on the plane containing $e_k$ and $e_{2k}$. The intersection of this plane with $M_n$ is square with vertexes $0$, $e_k,$ $e_{2k}$ and $e_k+e_{2k}$. Applying lemma \ref{fool} to the points $(j^k,\frac12)$ and $(\frac12,l^{2k})$, for large enough $n$ we have
\begin{multline*}
d(j(M_{2k}),l(M_{2k}))=\frac{\|j(M_{2k})-l(M_{2k})\|}{\vol_1(s\cap M_{2k})}=\frac{|j^k-\frac12|\,|l^{2k}-\frac12|}{|\frac14-l^{2k}j^k|}=\\
=\frac{(1-\frac2k+o\left(\frac1k\right))(1-\frac2k+o\left(\frac1k\right))}{1+o\left(\frac1k\right)}=1-\frac4k+o\left(\frac1k\right).
\end{multline*}

In case $n=2k+1$ the prove of the statement is similar by construction
$$
M_{2k+1}=M_{k-1}^{(1)}\times M_{k}^{(2)}.
$$ 
\end{proof}

Using the same argument for convex bodies
\begin{align*}
F_{2k}&=F_{k-1}^{(1)}\times F_{k-1}^{(2)},\quad F_{2k+1}=F_{k-1}^{(1)}\times F_{k}^{(2)}\qquad\qquad\qquad\qquad\qquad\\
\text{and}\qquad\qquad\qquad\qquad\qquad\qquad&\\
W_{2k}&=W_{k-1}^{(1)}\times W_{k-1}^{(2)},\quad W_{2k+1}=W_{k-1}^{(1)}\times W_{k}^{(2)}\\
\end{align*}
and taking into account lemmas \ref{lemmaCentroidJohn}, \ref{lemmaJohnCentroid}, \ref{lemmaCentroidLowner},  and  \ref{lemmaLownerCentroid} we can provide extremal convex bodies for measures of symmetry $\varphi_{g,l}$ and $\varphi_{g,j}$
\begin{align*}
d(g(F_n),j(F_n))=1-\frac1n\left(4+\frac{2e}{e-1}\left(1+\frac{\pi}2+\frac{\displaystyle e^{-\frac{\pi}{4}}}{\displaystyle\erf\left(\sqrt\pi/2\right)+1}\right)\right)+o\left(\frac1n\right),
\end{align*}
and
\begin{align*}
d(g(W_n),l(W_n))=1-\frac1n\left(4+\frac{2e}{(e-1)\left(1-\sqrt{\frac2\pi}\right)}\right)+o\left(\frac1n\right)
\end{align*}
where  $4+\frac{2e}{e-1}\left(1+\frac{\pi}2+\frac{\displaystyle e^{-\frac{\pi}{4}}}{\displaystyle\erf\left(\sqrt\pi/2\right)+1}\right)\approx 13$ and $4+\frac{2e}{(e-1)\left(1-\sqrt{\frac2\pi}\right)}\approx 20.$

%%СПИСОК ЛИТЕРАТУРЫ

\end{document}